\documentclass[letterpaper,12pt]{article}

\usepackage[utf8]{inputenc}
\usepackage[dvipsnames]{xcolor}
\usepackage[utf8]{inputenc}
\usepackage{amssymb}
\usepackage{amsmath} 
\usepackage[colorinlistoftodos]{todonotes}
\usepackage{amsthm}
\usepackage[margin=1in]{geometry}
\usepackage{fancyhdr}
\usepackage{upgreek}
\usepackage{cite}
\usepackage{enumitem}
\numberwithin{equation}{section}

\usepackage[colorlinks]{hyperref}

\title{New directions in fixed point theory in $G$-metric spaces and applications to  mappings contracting perimeters of triangles}

\author{Mohamed Jleli\footnote{Department of Mathematics, College of Science, King Saud University, Riyadh 11451, Saudi Arabia; E-mail address: jleli@ksu.edu.sa}, Cristina Maria P\u{a}curar\footnote{Faculty of Mathematics and Computer Science, Transilvania University of Bra\c{s}ov, 50 Iuliu Maniu Blvd., Bra\c{s}ov, Romania; E-mail address: cristina.pacurar@unitbv.ro},  Bessem Samet\footnote{Department of Mathematics, College of Science, King Saud University, Riyadh 11451, Saudi Arabia; E-mail address: bsamet@ksu.edu.sa}}

\date{}

\newtheorem{theorem}{Theorem}[section]
\newtheorem{lemma}[theorem]{Lemma}

\theoremstyle{definition}
\newtheorem{definition}[theorem]{Definition}
\newtheorem{example}[theorem]{Example}
\newtheorem{proposition}[theorem]{Proposition}
\newtheorem{corollary}[theorem]{Corollary}

\theoremstyle{remark}
\newtheorem{remark}[theorem]{Remark}

\DeclareMathOperator{\Fix}{Fix}

\begin{document}

\setlength{\headheight}{13.59999pt}
\addtolength{\topmargin}{-1.59999pt}
\maketitle

\begin{abstract}
We are concerned with the study of fixed points for mappings $T: X\to X$, where $(X,G)$ is a $G$-metric space in the sense of Mustafa and Sims. After the publication of the paper [Journal of Nonlinear and Convex Analysis. 7(2) (2006) 289--297] by  Mustafa and Sims, a great interest was devoted to the study of fixed points in  $G$-metric spaces. In 2012, the first and third authors observed that several fixed point theorems established in $G$-metric spaces are immediate consequences   of known fixed point theorems in standard metric spaces.   This observation demotivated the investigation of fixed points in $G$-metric spaces. In this paper, we open new directions  in fixed point theory in $G$-metric spaces. Namely, we establish new versions of the Banach, Kannan and Reich fixed point theorems in $G$-metric spaces. We point out that the approach used by the first and third authors [Fixed Point Theory Appl. 2012 (2012) 1--7] is inapplicable in the present study.  We also provide some interesting applications related to mappings contracting perimeters of triangles.
\vspace{0.1cm}

\end{abstract}

\section{Introduction}\label{sec1}

The notion of $G$-metric spaces was introduced by Mustafa and Sims in \cite{MU}. After the publication of this work, a great attention was accorded to the study of fixed points in such spaces, see e.g. \cite{AYP,AY,KA,MUS,MU2,RED,SAL,SH} and the references therein.  In \cite{JS} (see also \cite{Samet}), Jleli and Samet remarked that several fixed point theorems  in $G$-metric spaces can be deduced immediately from fixed point theorems in (standard) metric spaces. Namely, it was shown that, if $G: X\times X\times X\to \mathbb{R}^+$ is a $G$-metric on $X$, then the mapping $\delta: X\times X\to \mathbb{R}^+$ defined by $\delta(x,y)=\max\{G(x,x,y),G(y,x,x)\}$ is a metric on $X$. Moreover, if $(X,G)$ is a complete $G$-metric space, then $(X,\delta)$ is a complete metric space. From this observation, it was proved that several contractions studied in complete $G$-metric spaces can be reduced to standard contractions in the complete metric space $(X,\delta)$. For instance, if $(X,G)$ is a complete $G$-metric space and $T: X\to X$ is a mapping satisfying the inequality $G(Tx,Ty,Tz)\leq \lambda G(x,y,z)$ for all $x,y,z\in X$, where $\lambda\in (0,1)$ is a constant, then (see \cite{MUS})  $T$ admits a unique fixed point. Taking $z=y$, the above inequality reduces to $G(Tx,Ty,Ty)\leq \lambda G(x,y,y)$ for all $x,y\in X$. Replacing $(x,y,z)$ by $(y,x,x)$, we obtain $G(Ty,Tx,Tx)\leq \lambda G(y,x,x)$ for all $x,y\in X$. Consequently, we get $\delta(Tx,Ty)\leq \lambda \delta(x,y)$ for all $x,y\in X$, and the Banach fixed point theorem (in metric spaces) applies. After the publications of the papers \cite{JS,Samet}, the attention paid to the study of fixed point theorems in $G$-metric spaces  was attenuated. 

In this paper, we open new directions in fixed point theory  on $G$-metric spaces. Namely, we establish new extensions of Banach, Kannan and Reich fixed point theorems in $G$-metric spaces. We point out that the used approach  in \cite{JS,Samet} cannot be used to establish our obtained results.  Moreover, motivated by the recent paper \cite{Petrov}, we provide 
(among other applications) an application  to the study of fixed points contracting perimeters of triangles. 

The rest of the paper is organized as follow. In Section \ref{sec2}, we briefly recall some basic notions and properties related to $G$-metric spaces. Our main results are stated and proved in Section \ref{sec3}. Finally, some applications are provided in Section \ref{sec4}. 

Throughout this paper, the following notations are used:  $\mathbb{R}^+=[0,\infty)$, $X$ denotes a nonempty set,  $|X|$ denotes the cardinal of $X$, and for a mapping $T: X\to X$, the set of fixed points of $T$ is denoted by $\Fix(T)$.

\section{Preliminaries}\label{sec2}

Let us recall briefly some basic notions related to $G$-metric spaces. For more details, we refer to Mustafa and Sims \cite{MU}. Throughout this paper, $X$ denotes a nonempty set and $\mathbb{R}^+=[0,\infty)$.  

Let $G: X\times X\times X\to \mathbb{R}^+$ be a given mapping. We say that $G$ is a $G$-metric on $X$, if for all $x,y,z,w\in X$, we have
\begin{itemize}
\item[(P$_1$)] $G(x,y,z)=0$ if and only if $x=y=z$;
\item[(P$_2$)] If $x\neq y$, then $G(x,x,y)>0$;
\item[(P$_3$)] 	$G(x,y,z)=G(\sigma(x,y,z))$ for every permutation $\sigma:\{x,y,z\}\to \{x,y,z\}$;
\item[(P$_4$)] If $y\neq z$, then $G(x,x,y)\leq G(x,y,z)$;
\item[(P$_5$)] $G(x,y,z)\leq G(x,w,w)+G(w,y,z)$. 
\end{itemize}
If the above conditions are satisfied, then $(X,G)$ is called a $G$-metric space. 

Prototype examples  of $G$-metrics are
\begin{equation}\label{G-d}
G(x,y,z)=d(x,y)+d(y,z)+d(x,z),\quad x,y,z\in X
\end{equation}
and
\begin{equation}\label{G2-d}
G(x,y,z)=\max\left\{d(x,y),d(y,z),d(x,z)\right\},\quad x,y,z\in X,
\end{equation}
where $d$ is a (standard) metric on $X$. 

Let $\{x_n\}$ be a sequence in $X$ and $x\in X$. We say that $\{x_n\}$ is $G$-convergent to $x$, if $\displaystyle\lim_{n,m\to \infty}G(x_n,x_m,x)=0$. The following assertions are equivalent:
\begin{itemize}
\item[(A$_1$)] $\{x_n\}$ is $G$-convergent to $x$;
\item[(A$_2$)] $\displaystyle\lim_{n\to \infty}G(x_n,x_n,x)=0$;	
\item[(A$_3$)] $\displaystyle\lim_{n\to \infty}G(x_n,x,x)=0$. 
\end{itemize}
Notice that, if $\{x_n\}$ is $G$-convergent to $x$ and $\{x_n\}$ is $G$-convergent to $y$, then $x=y$. This can be easily seen by (P$_5$), (A$_2$) and (A$_3$).

We say that $\{x_n\}$ is $G$-Cauchy, if $\displaystyle\lim_{n,m\to \infty}G(x_n,x_m,x_m)=0$. 

If every  $G$-Cauchy sequence in $X$ is $G$-convergent to an element of $X$, then $(X,G)$ is called a complete $G$-metric space.

The following two lemmas will be useful later. 
\begin{lemma}\label{L2.1}
Let $d$ be a metric on $X$ and $G: X\times X\times X\to \mathbb{R}^+$ be the $G$-metric on $X$ defined by \eqref{G-d}. 	The following assertions are equivalent:
\begin{itemize}
\item[{\rm{(i)}}] $(X,d)$ is a complete metric space;
\item[{\rm{(ii)}}] $(X,G)$ is a complete $G$-metric space.
	\end{itemize}
\end{lemma}

\begin{proof}
Assume first that $(X,d)$ is a complete metric space. Let $\{x_n\}\subset X$ be a $G$-Cauchy sequence. For all $n,m$, we have
$$
\begin{aligned}
d(x_n,x_m)& \leq 2 d(x_n,x_m)\\
&= d(x_n,x_m)+d(x_m,x_m)+d(x_n,x_m)\\
&=G(x_n,x_m,x_m). 
\end{aligned}
$$	
Since $\{x_n\}$ is a $G$-Cauchy sequence, we have $\displaystyle\lim_{n,m\to \infty}G(x_n,x_m,x_m)=0$, which implies by the above inequality that $\displaystyle\lim_{n,m\to \infty}d(x_n,x_m)=0$, that is, $\{x_n\}$ is a Cauchy sequence in the metric space $(X,d)$. From the completeness of $(X,d)$, we deduce the existence of $x\in X$ such that 
\begin{equation}\label{lim}
\lim_{n\to \infty}d(x_n,x)=0.
\end{equation}
On the other hand, we have
$$
\begin{aligned}
G(x_n,x,x)&=d(x_n,x)+d(x,x)+d(x_n,x)\\
&=2 d(x_n,x),
\end{aligned}
$$
which implies by \eqref{lim} that $\displaystyle\lim_{n\to \infty}G(x_n,x,x)=0$,
that is, $\{x_n\}$ is $G$-convergent to $x$. This shows that $(X,G)$ is a complete $G$-metric space.

Assume now that $(X,G)$ is a complete $G$-metric space.  Let $\{x_n\}\subset X$ be a Cauchy sequence in the metric space $(X,d)$. From the identity
$$
G(x_n,x_m,x_m)=\frac{1}{2}d(x_n,x_m), 
$$
we deduce that $\displaystyle\lim_{n,m\to \infty}G(x_n,x_m,x_m)=0$, that is, $\{x_n\}$ is $G$-Cauchy. Consequently, there exists $x\in X$ such that 
$$
\lim_{n\to \infty}G(x_n,x,x)=0.
$$
Then, from the identity 
$$
d(x_n,x)=\frac{1}{2}G(x_n,x,x), 
$$
we deduce that $\displaystyle \lim_{n\to \infty}d(x_n,x)=0$, which shows that $(X,d)$ is a complete metric space. 
\end{proof}

\begin{lemma}[see \cite{JS}]\label{L2.2}
Let $d$ be a metric on $X$ and $G: X\times X\times X\to \mathbb{R}^+$ be the $G$-metric on $X$ defined by \eqref{G2-d}. 	The following assertions are equivalent:
\begin{itemize}
\item[{\rm{(i)}}] $(X,d)$ is a complete metric space;
\item[{\rm{(ii)}}] $(X,G)$ is a complete $G$-metric space.
	\end{itemize}
\end{lemma}

\section{Main results}\label{sec3}

\subsection{Banach's fixed point theorem in $G$-metric spaces}

Our first main result is an extension of the Banach fixed point theorem to $G$-metric spaces. 

\begin{theorem}\label{T3.1}
Let $(X,G)$ be a complete $G$-metric space	with $|X|\geq 3$. Let $T: X\to X$ be a mapping satisfying the following conditions:
\begin{itemize}
\item[{\rm{(I)}}] For all $x\in X$, $T(Tx)\neq x$, provided $Tx\neq x$;
\item[{\rm{(II)}}]	There exists $\lambda\in (0,1)$ such that for all  pairwise distinct points $x,y,z\in X$, we have
$$
G(Tx,Ty,Tz)\leq \lambda G(x,y,z).
$$
\end{itemize}
Then $\Fix(T)\neq\emptyset$ and $|\Fix(T)|\leq 2$. 
\end{theorem}

\begin{proof}
Let us first prove that $\Fix(T)\neq\emptyset$. Let $u_0\in X$ be fixed and consider the Picard sequence $\{u_n\}\subset X$ defined by 
$$
u_{n+1}=Tu_n,\quad n\geq 0. 
$$
If $u_n=u_{n+1}$ for some $n$, then $u_n\in \Fix(T)$ and the result is proved. So, we may suppose that 
$$
u_n\neq u_{n+1},\quad n\geq 0, 
$$
which implies by (I) that $u_n\neq T(Tu_n)$ ($=u_{n+2}$).  Consequently,  $u_n, u_{n+1}$ and $u_{n+2}$ are pairwise distinct points for every $n\geq 0$. Then, making use of (II), we obtain
$$
\begin{aligned}
& G(u_1,u_2,u_3)=G(Tu_0,Tu_1,Tu_2)	\leq \lambda G(u_0,u_1,u_2),\\
& G(u_2,u_3,u_4)=G(Tu_1,Tu_2,Tu_3)	\leq \lambda G(u_1,u_2,u_3)\leq \lambda^2 G(u_0,u_1,u_2),\\
&\vdots \\
&G(u_n,u_{n+1},u_{n+2})=G(Tu_{n-1},Tu_n,Tu_{n+1})	\leq \lambda G(u_{n-1},u_n,u_{n+1})\leq \lambda^n G(u_0,u_1,u_2),
\end{aligned}
$$
that is,
\begin{equation}\label{S1-T1}
G(u_n,u_{n+1},u_{n+2})\leq \tau_0\lambda^n,\quad n\geq 0,	
\end{equation}
where $\tau_0=G(u_0,u_1,u_2)>0$ (by (P$_1$)). We now show that $\{u_n\}$ is $G$-Cauchy. First, by (P$_3$), (P$_4$) and using that $u_n, u_{n+1}$ and $u_{n+2}$ are pairwise distinct points, we obtain 
$$
\begin{aligned}
G(u_n,u_{n+1},u_{n+1})&=G(u_{n+1},u_{n+1},u_{n})\\
&\leq G(u_{n+1},u_n,u_{n+2})\\
&=G(u_n,u_{n+1},u_{n+2}),
\end{aligned}
$$
which implies by \eqref{S1-T1} that 
\begin{equation}\label{S2-T1}
G(u_n,u_{n+1},u_{n+1})\leq \tau_0\lambda^n,\quad n\geq 0. 	
\end{equation}
We now use \eqref{S2-T1} and (P$_5$) to obtain that for all $n<m$,
$$
\begin{aligned}
G(u_n,u_m,u_m) &\leq G(u_n,u_{n+1},u_{n+1})+G(u_{n+1},u_m,u_m)\\
&\leq 	G(u_n,u_{n+1},u_{n+1})+G(u_{n+1},u_{n+2},u_{n+2})+G(u_{n+2},u_m,u_m)\\
&\vdots\\
&\leq G(u_n,u_{n+1},u_{n+1})+G(u_{n+1},u_{n+2},u_{n+2})+\cdots+G(u_{m-1},u_m,u_m)\\
&\leq \tau_0\left(\lambda^n+\lambda^{n+1}+\cdots +\lambda^{m-1}\right) \\
&=\frac{\tau_0}{1-\lambda}\lambda^n\left(1-\lambda^{m-n}\right)\\
&\leq \frac{\tau_0}{1-\lambda}\lambda^n, 
\end{aligned}
$$
which shows that (since $0<\lambda<1$) 
$$
\lim_{n,m\to \infty} G(u_n,u_m,u_m)=0,
$$
that is, $\{u_n\}$ is a $G$-Cauchy sequence. Now, the completeness of $(X,G)$ yields the existence of an element $u^*\in X$ such that $\{u_n\}$ is $G$-convergent to $u^*$.
On the other hand, if there exists $k$ such that for all $n\geq k$, we have $u_n=u^*$, this contradicts the fact that   $u_n, u_{n+1}$ and $u_{n+2}$ are pairwise distinct points for all $n$. Consequently, we can extract a subsequence $\{u_{n(k)}\}_k$ of $\{u_n\}$ such  that  
$u_{n(k)}\neq u^*$ for all $k$. In order to simplify writing, we always denote the sequence $\{u_{n(k)}\}_k$ by $\{u_n\}$ with $u_n\neq u^*$ for all $n$. We now show that $u^*\in \Fix(T)$. Indeed,   by (P$_3$) and (P$_5$), we have
$$
\begin{aligned}
G(u^*,u^*,Tu^*) &\leq G(u^*,u_n,u_n)+G(u_n,u^*,Tu^*)	\\
&=G(u^*,u_n,u_n)+G(u^*,u_n,Tu^*)\\
&\leq G(u^*,u_n,u_n)+G(u^*,u_{n+1},u_{n+1})+G(u_{n+1},u_n,Tu^*)\\
&= G(u^*,u_n,u_n)+G(u^*,u_{n+1},u_{n+1})+G(Tu_{n},Tu_{n-1},Tu^*),
\end{aligned}
$$
which implies by (II) and the fact that $u_{n},u_{n-1}$ and $u^*$ are pairwise distinct points that 
\begin{equation}\label{S3-T1}
G(u^*,u^*,Tu^*)\leq G(u^*,u_n,u_n)+G(u^*,u_{n+1},u_{n+1})+\lambda G(u_n,u_{n-1},u^*). 
\end{equation}
On the other hand, by  the definition of the $G$-convergence and (A$_2$), we have
$$
\lim_{n\to \infty} \left[G(u^*,u_n,u_n)+G(u^*,u_{n+1},u_{n+1})+\lambda G(u_n,u_{n-1},u^*)\right]=0,
$$
which implies by \eqref{S3-T1} together with (P$_1$) that $G(u^*,u^*,Tu^*)=0$ and $u^*=Tu^*$. This proves that $\Fix(T)\neq \emptyset$.

Suppose now that $v_i$, $i=1,2,3$, are  pairwise distinct fixed points of $T$.   Then, making use of (P$_1$) and (II), we get
$$
0<G(v_1,v_2,v_3)=G(Tv_1,Tv_2,Tv_3)\leq \lambda G(v_1,v_2,v_3)
$$
and we reach a contradiction with $0<\lambda<1$. This shows that $|\Fix(T)|\leq 2$. The proof of Theorem \ref{T3.1} is then completed.
\end{proof}

\begin{remark}
We point out that the condition  $x,y,z \in X$ are pairwise distinct in assumption (II) of Theorem \ref{T3.1} is essential. Otherwise, as we mentioned in Section \ref{sec1}, using the approach in \cite{JS}, (II) reduces to 
$$
\delta(Tx,Ty)\leq \lambda \delta(x,y),\quad x,y\in X,
$$
where $\delta: X\times X\to \mathbb{R}^+$ is the metric on $X$ defined by 
\begin{equation}\label{delt-f}
\delta (x,y)=\max\{G(x,y,y),G(y,x,x)\},\quad x,y\in X,
\end{equation}
and the existence of a (unique) fixed point of $T$ will be an immediate consequence 	of the Banach fixed point theorem in metric spaces. 
\end{remark}

We provide below an example to illustrate our obtained result.

\begin{example}\label{ex3.3}
Let $X=\{A,B,C\}\subset \mathbb{R}^2$, where 
$$
A=\left(\frac{7}{8},\frac{\sqrt{15}}{8}\right),\quad B=(1,0),\quad C=(0,0).
$$
We introduce the mappings $T_1,T_2:X\to X$ defined by 
$$
T_1A=A,\quad T_1B=B,\quad T_1C=A
$$
and
$$
T_2A=B,\quad T_2B=A,\quad T_2C=A.
$$
Consider now the $G$-metric on $X$ given by 
$$
G(u,v,w)=\max\left\{\|u-v\|,\|v-w\|,\|u-w\|\right\},\quad u,v,w\in X,
$$
where $\|\cdot\|$ denotes the Euclidean norm in $\mathbb{R}^2$. 

Clearly, the mapping $T_1$ satisfies condition (I) of Theorem \ref{T3.1}. Furthermore, we have 
$$
\begin{aligned}
G(T_1A,T_1B,T_1C)&=G(A,B,A)\\
&=\|A-B\|\\
&=\frac{1}{2}
\end{aligned}
$$
and
$$
\begin{aligned}
G(A,B,C)&=\max\left\{\|A-B\|,\|B-C\|,\|A-C\|\right\}\\
&=\max\left\{\frac{1}{2}, 1, 1\right\}\\
&=1,
\end{aligned}
$$
which shows that condition (II) of Theorem \ref{T3.1} holds for all $\lambda\in \left[\frac{1}{2},1\right)$. On the other hand, $\Fix(T_1)=\{A,B\}$, which confirms the obtained result given by Theorem \ref{T3.1}. 

Remark also that the mapping $T_2$ satisfies  condition (II) of Theorem \ref{T3.1}. This can be easily seen observing that $G(T_2A,T_2B,T_2C)=\|A-B\|=G(T_1A,T_1B,T_1C)$. On the other hand, we have $T_2A\neq A$ and $T_2(T_2A)=T_2(B)=A$, which shows that condition (I) of Theorem \ref{T3.1} is not satisfied. Furthermore, we have $\Fix(T)=\emptyset$, which shows that in the absence of   condition (I), the result of Theorem \ref{T3.1} is not true.  
\end{example}

\subsection{Kannan's fixed point theorem in $G$-metric spaces}

The Kannan fixed point theorem \cite{KAN} can be stated as follows:  Let $(X,d)$ be a complete metric space and $T: X\to X$ be a mapping satisfying the inequality
$$
d(Tx,Ty)\leq \lambda \left[d(x,Tx)+d(y,Ty)\right]
$$
for every $x,y\in X$, where $\lambda\in \left(0,\frac{1}{2}\right)$ is a constant.  Then $T$ admits a unique fixed point. 

Our second main result is an extension of the above result from metric spaces to $G$-metric spaces. 

\begin{theorem}\label{T3.4}
Let $(X,G)$ be a complete $G$-metric space	with $|X|\geq 3$. Let $T: X\to X$ be a mapping satisfying the following conditions:
\begin{itemize}
\item[{\rm{(I)}}] For all $x\in X$, $T(Tx)\neq x$, provided $Tx\neq x$;
\item[{\rm{(II)}}]	There exists $\lambda\in \left(0,\frac{1}{3}\right)$ such that for all  pairwise distinct points $x,y,z\in X$, we have
$$
G(Tx,Ty,Tz)\leq \lambda \left[G(x,Tx,Tx)+G(y,Ty,Ty)+G(z,Tz,Tz)\right].
$$
\end{itemize}
Then $\Fix(T)\neq\emptyset$ and $|\Fix(T)|\leq 2$. 
\end{theorem}

\begin{proof}
We first show that $T$ admits at least one fixed point. Let $u_0\in X$ be fixed and consider the Picard sequence $\{u_n\}\subset X$ defined by 
$$
u_{n+1}=Tu_n,\quad n\geq 0. 
$$
As in the proof of Theorem \ref{T3.1}, we may suppose that $u_n\neq u_{n+1}$ for all $n\geq 0$, 
which implies by (I) that   $u_n, u_{n+1}$ and $u_{n+2}$ are pairwise distinct points for every $n\geq 0$. Then, making use of (II), we obtain
\begin{eqnarray*}
G(u_1,u_2,u_3)&=&G(Tu_0,Tu_1,Tu_2)\\
 &\leq & \lambda \left[G(u_0,u_1,u_1)+G(u_1,u_2,u_2)+G(u_2,u_3,u_3)\right],
\end{eqnarray*}
which yields
\begin{equation}\label{SS1}
(1-\lambda) G(u_1,u_2,u_3)\leq 	\lambda \left[G(u_0,u_1,u_1)+G(u_2,u_3,u_3)\right].
\end{equation}
We now use (P$_3$),  (P$_4$), (P$_5$) and the fact that $u_n, u_{n+1}$ and $u_{n+2}$ are pairwise distinct points  to get 
\begin{equation}\label{SS2}
G(u_0,u_1,u_1)=G(u_1,u_1,u_0)\leq G(u_1,u_0,u_2)=G(u_0,u_1,u_2)	
\end{equation}
and
\begin{equation}\label{SS3}
G(u_2,u_3,u_3)=G(u_3,u_3,u_2)\leq G(u_3,u_1,u_2)=G(u_1,u_2,u_3). 
\end{equation}
Then, it follows from \eqref{SS1}, \eqref{SS2} and \eqref{SS3} that 
$$
(1-2\lambda) G(u_1,u_2,u_3)\leq\lambda G(u_0,u_1,u_2),
$$
which implies (since $1-2\lambda>0$) that 
$$
G(u_1,u_2,u_3)\leq k  G(u_0,u_1,u_2),
$$
where $k=\frac{\lambda}{1-2\lambda}$. Notice that by $0<\lambda<\frac{1}{3}$, we have $0<k<1$. Next, by induction, we obtain easily that 
$$
G(u_n,u_{n+1},u_{n+2})\leq k^n \tau_0, \quad n\geq 0,
$$
where $\tau_0=G(u_0,u_1,u_2)>0$. Proceeding as in the proof of Theorem \ref{T3.1}, we obtain that $\{u_n\}$ is a $G$-Cauchy sequence.  Due to the completeness of $(X,G)$, there exists $u^*\in X$ such that $\{u_n\}$ is $G$-convergent to $u^*$.  By the proof of Theorem \ref{T3.1}, without restriction of the generality, we may suppose that $u_n\neq u^*$ for all $n$. We now show that $u^*\in \Fix(T)$. By (P$_3$) and (P$_4$), we have (see the proof of Theorem \ref{T3.1}) 
$$
G(u^*,u^*,Tu^*)\leq G(u^*,u_n,u_n)+G(u^*,u_{n+1},u_{n+1})+G(Tu_{n},Tu_{n-1},Tu^*),
$$
which implies by (II) and the fact that $u_{n},u_{n-1}$ and $u^*$ are pairwise distinct points that 
$$
\begin{aligned}
G(u^*,u^*,Tu^*)&\leq G(u^*,u_n,u_n)+G(u^*,u_{n+1},u_{n+1})\\
&\quad +\lambda \left[G(u_n,u_{n+1},u_{n+1})+G(u_{n-1},u_n,u_n)+G(u^*,Tu^*,Tu^*)\right] \\
&\leq G(u^*,u_n,u_n)+G(u^*,u_{n+1},u_{n+1})\\
&\quad +\lambda \left[G(u_n,u_{n+1},u_{n+1})+G(u_{n-1},u_n,u_n)+2G(u^*,u^*,Tu^*)\right],
\end{aligned}
$$
that is,
$$
\begin{aligned}
(1-2\lambda)G(u^*,u^*,Tu^*)&\leq  G(u^*,u_n,u_n)+G(u^*,u_{n+1},u_{n+1})\\
&\quad +\lambda \left[G(u_n,u_{n+1},u_{n+1})+G(u_{n-1},u_n,u_n)\right].
\end{aligned}
$$
Since $1-2\lambda>0$, passing to the limit as $n\to \infty$ in the above inequality, $G(u^*,u^*,Tu^*)=0$, which yields $u^*=Tu^*$. This shows that $\Fix(T)\neq\emptyset$.  

Suppose now that $v_i$, $i=1,2,3$, are  pairwise distinct fixed points of $T$.   Then, making use of (II), we get
\begin{equation}\label{SSI}
\begin{aligned}
G(v_1,v_2,v_3)&=G(Tv_1,Tv_2,Tv_3)\\
&\leq \lambda\left[G(v_1,Tv_1,Tv_1)+G(v_2,Tv_2,Tv_2)+G(v_3,Tv_3,Tv_3)\right]\\
&=\lambda\left[G(v_1,v_1,v_1)+G(v_2,v_2,v_2)+G(v_3,v_3,v_3)\right].
\end{aligned}
\end{equation}
On the other hand,  we have by (P$_4$) that 
$$
\begin{aligned}
&G(v_1,v_1,v_1)\leq G(v_1,v_1,v_2)\leq G(v_1,v_2,v_3)\\
&G(v_2,v_2,v_2)\leq G(v_2,v_2,v_3)\leq G(v_2,v_3,v_1)=G(v_1,v_2,v_3)\\
&G(v_3,v_3,v_3)\leq G(v_3,v_3,v_2)\leq G(v_3,v_2,v_1)=G(v_1,v_2,v_3).
 \end{aligned}
$$
Then, by \eqref{SSI}, we obtain
$$
(1-3\lambda)G(v_1,v_2,v_3)\leq 0,
$$
and we reach a contradiction with $0<\lambda<\frac{1}{3}$ and $G(v_1,v_2,v_3)>0$. This shows that $|\Fix(T)|\leq 2$ and the proof of Theorem \ref{T3.4} is completed. 
\end{proof}

\begin{remark}
Notice that, if conditions (II) of Theorem \ref{T3.4} holds for every $x,y,z\in X$, then following the approach in \cite{JS}, the existence of  a (unique) fixed point of $T$ can be deduced  	immediately from  \'Ciri\'c's fixed point theorem in metric spaces \cite{Ciric}. Namely, in this case, taking $z=y$, (II) reduces to 
$$
G(Tx,Ty,Ty)\leq \lambda \left[G(x,Tx,Tx)+2G(y,Ty,Ty)\right],\quad x,y\in X.
$$
Replacing  $y$ by $x$ and $z$ by $y$, (II) reduces to 
$$
G(Tx,Tx,Ty)\leq \lambda\left[2G(x,Tx,Tx)+G(y,Ty,Ty)\right],\quad x,y\in X.
$$
The above two  inequalities yield
$$
\delta(Tx,Ty)\leq k\max\left\{\delta(x,Tx),\delta(y,Ty)\right\},
$$
where $k=3\lambda$ and $\delta$ is the metric on $X$ defined by \eqref{delt-f}.
Since $k=3\lambda\in (0,1)$, then by \'Ciri\'c's fixed point theorem, $T$ admits a (unique) fixed point. 
\end{remark}

\begin{example}\label{ex3.5}
Let $X=\left\{a,b,c\right\}\subset \mathbb{R}$, where 
$$
a=0,\quad b=\frac{1}{5},\quad c=1.
$$ 
We introduce the mappings $T_1,T_2:X\to X$ defined by 
$$
T_1a=a,\quad T_1b=b,\quad T_1c=a
$$
and
$$
T_2a=b,\quad T_2b=a,\quad T_2c=a.
$$
Consider now the $G$-metric on $X$ given by 
$$
G(u,v,w)=\max\left\{|u-v|,|v-w|,|u-w|\right\},\quad u,v,w\in X.
$$
Clearly, the mapping $T_1$ satisfies condition (I) of Theorem \ref{T3.4}. Furthermore, we have 
$$
G(T_1a,T_1b,T_1c)=G(a,b,a)=|a-b|=\frac{1}{5}
$$
and
$$
\begin{aligned}
G(a,T_1a,T_1a)+G(b,T_1b,T_1b)+G(c,T_1c,T_1c)&=G(a,a,a)+G(b,b,b)+G(c,a,a)\\
&=G(c,a,a)\\
&=|a-c|=1,
\end{aligned}
$$
which shows that condition (II) of Theorem \ref{T3.4} holds for all $\lambda\in \left(0,\frac{1}{3}\right)$. On the other hand, $\Fix(T_1)=\{a,b\}$, which confirms the obtained result given by Theorem \ref{T3.4}.

Remark also that the mapping $T_2$ satisfies  condition (II) of Theorem \ref{T3.4}. Namely, we have 
$$
G(T_2a,T_2b,T_2c)=G(b,a,a)=|b-a|=\frac{1}{5}
$$
and
$$
\begin{aligned}
G(a,T_2a,T_2a)+G(b,T_2b,T_2b)+G(c,T_2c,T_2c)&=G(a,b,b)+G(b,a,a)+G(c,a,a)\\
&=2|a-b|+|c-a|\\
&=\frac{2}{5}+1=\frac{7}{5},
\end{aligned}
$$
which shows that condition (II) of Theorem \ref{T3.4} holds for every $\lambda\in \left[\frac{1}{7},\frac{1}{3}\right)$. On the other hand, we have $T_2a\neq a$ and $T_2(T_2a)=T_2b=a$, which shows that condition (I) of Theorem \ref{T3.4} is not satisfied. Furthermore, we have $\Fix(T)=\emptyset$, which shows that in the absence of   condition (I), the result of Theorem \ref{T3.4} is not true.  
\end{example}

\subsection{The  Reich fixed point theorem in $G$-metric spaces}

The following fixed point result was obtained by Reich \cite{RE}:  Let $(X,d)$ be a complete metric space and $T: X\to X$ be a mapping satisfying the inequality
$$
d(Tx,Ty)\leq a_1 d(x,Tx)+a_2 d(y,Ty)+a_3 d(x,y)
$$
for every $x,y\in X$, where $a_i$ ($i=1,2,3$) are nonnegative constants with $0<\sum_{i=1}^3a_i<1$. Then $T$ admits a unique fixed point. 

We now extend the Reich fixed point theorem  to $G$-metric spaces.

\begin{theorem}\label{T3.7}
Let $(X,G)$ be a complete $G$-metric space	with $|X|\geq 3$. Let $T: X\to X$ be a mapping satisfying the following conditions:
\begin{itemize}
\item[{\rm{(I)}}] For all $x\in X$, $T(Tx)\neq x$, provided $Tx\neq x$;
\item[{\rm{(II)}}]	There exist nonnegative constants $a_i$ ($i=1,2,3,4$) with $0<\sum_{i=1}^4 a_i<1$  such that for all  pairwise distinct points $x,y,z\in X$, we have
$$
G(Tx,Ty,Tz)\leq a_1G(x,Tx,Tx)+a_2G(y,Ty,Ty)+a_3G(z,Tz,Tz)+a_4G(x,y,z).
$$
\end{itemize}
Then $\Fix(T)\neq\emptyset$ and $|\Fix(T)|\leq 2$. 
\end{theorem}

\begin{proof}
Let us first show that $T$ admits a fixed point. Let $u_0 \in X$ be fixed, but arbitrarily chosen and consider the Picard sequence $\{u_n\} \subset X$ defined by $$u_{n+1} = Tu_n, \quad n\geq 0$$.

As in the proof of Theorem \ref{T3.1}, we can assume that $u_n \neq u_{n+1}$ for all $n \geq 0$, and thus, by \rm{(I)} we obtain that $u_n, u_{n+1}, u_{n+2}$ are pairwise distinct for all $n \geq 0$. Then, by \rm{(II)}, we have
$$
\begin{aligned}
	G(u_1,u_2,u_3) &= G(Tu_0,Tu_1,Tu_2)\\
	 &\leq  a_1G(u_0,Tu_0,Tu_0)+a_2G(u_1,Tu_1,Tu_1)+a_3G(u_2,Tu_2,Tu_2)+a_4G(u_0,u_1,u_2)
\end{aligned}
$$
so
\begin{equation}
	\begin{aligned}
		G(u_1,u_2,u_3) \leq a_1G(u_0,u_1,u_1)+a_2G(u_1,u_2,u_2)+a_3G(u_2,u_3,u_3)+a_4G(u_0,u_1,u_2)
	\end{aligned}
	\label{3.7eq1}
\end{equation}
By (P$_3$) and (P$_4$), since $u_n, u_{n+1}, u_{n+2}$ are pairwise distinct for all $n \geq 0$ we have
\begin{equation}
	G(u_0,u_1,u_1) \leq G(u_0,u_1,u_2)
	\label{3.7eq2}
\end{equation}
\begin{equation}
	G(u_1,u_2,u_2) \leq G(u_0,u_1,u_2)
	\label{3.7eq3}
\end{equation}
\begin{equation}
	G(u_2,u_3,u_3) \leq G(u_1,u_2,u_3).
	\label{3.7eq4}
\end{equation}
By (\ref{3.7eq1}), (\ref{3.7eq2}), (\ref{3.7eq3}) and (\ref{3.7eq4}) it follows that
$$
\begin{aligned}
	(1-a_3)G(u_1,u_2,u_3) \leq  (a_1+a_2+a_4)G(u_0,u_1,u_2),
\end{aligned}
$$
and since $1-a_3 > 0$ we obtain
$$ G(u_1,u_2,u_3) \leq \dfrac{a_1+a_2+a_4}{1-a_3}G(u_0,u_1,u_2).$$ 
Let $k = \dfrac{a_1+a_2+a_4}{1-a_3}$. Since $\sum_{i=1}^4 a_i<1$ notice that $k <1$ and we have
$$ G(u_1,u_2,u_3) \leq k\tau_0,$$
where $\tau_0=  G(u_0,u_1,u_2) > 0$. By induction, we get 
$$ G(u_n,u_{n+1},u_{n+2}) \leq k^n\tau_0.$$
As in the proof of Theorem \ref{T3.1}, we obtain that $\{u_n\}$ is a G-Cauchy sequence and by completeness of $(X,G)$, there exists $u^* \in X$ such that $\{u_n\}$ is G-convergent to $u^*$. We will show that $u^*\in Fix(T)$. By (P$_5$), we have
$$\begin{aligned}
	G(u^*,Tu^*,Tu^*) &\leq G(u^*,u_n,u_n) + G(u_n, Tu^*, Tu^*)\\
	& = G(u^*,u_n,u_n) + G(Tu_{n-1}, Tu^*, Tu^*),
\end{aligned}$$ 
which implies, by \rm{(II)} and the fact that $u_n$, $u_{n-1}$ and $x^*$ are pairwise distinct that 
$$\begin{aligned}
	G(u^*,Tu^*,Tu^*) &\leq G(u^*,u_n,u_n) + a_1 G(u_{n-1}, Tu_{n-1}, Tu_{n-1}) \\&+ a_2 G(u^*, Tu^*, Tu^*) + a_3 G(u^*, Tu^*, Tu^*) + a_4 G(u_{n-1}, u^*, u^*),
\end{aligned}$$ 
so we have 
$$\begin{aligned}
	[1-(a_2+a_3)]G(u^*,Tu^*,Tu^*) &\leq G(u^*,u_n,u_n) + a_1 G(u_{n-1}, u_{n}, u_{n}) + a_4 G(u_{n-1}, u^*, u^*).
\end{aligned}$$ 
Since $1-(a_2+a_3) > 0$, passing to limit as $n \to \infty$ in the above inequality, we obtain $G(u^*,Tu^*,Tu^*) = 0$, which yields $u^*=Tu^*$, and thus Fix$(T) \neq \emptyset$.

Now suppose that $v_i$, $i =1,2,3$, are pairwise distinct points of $T$. Then, by \rm{(II)} we have 
	$$
	\begin{aligned}
		G(v_1,v_2,v_3) &= G(Tv_1,Tv_2,Tv_3)\\
		&\leq  a_1G(v_1,Tv_1,Tv_1)+a_2G(v_2,Tv_2,Tv_2)+a_3G(v_3,Tv_3,Tv_3)+a_4G(v_1,v_2,v_3)\\
		&= a_1G(v_1,v_1,v_1)+a_2G(v_2,v_2,v_2)+a_3G(v_3,v_3,v_3)+a_4G(v_1,v_2,v_3),
	\end{aligned}
	$$
that is 
	$$
\begin{aligned}
	(1-a_4)G(v_1,v_2,v_3) &\leq  a_1G(v_1,v_1,v_1)+a_2G(v_2,v_2,v_2)+a_3G(v_3,v_3,v_3) \\
	&=0.
\end{aligned}
$$
Thus we obtain $a_4\geq 1$ which is a contradiction. This shows that $|\Fix(T)|\leq 2$ and the proof of Theorem \ref{T3.7} is completed. 

\end{proof}

\begin{remark}
If condition (II) of Theorem \ref{T3.7} holds for every $x,y,z\in X$, then following the approach in \cite{JS}, the existence of  a (unique) fixed point of $T$ can be deduced  	immediately from  \'Ciri\'c's fixed point theorem in metric spaces \cite{Ciric} . Namely, taking $z=y$, (II) reduces to  
$$
\begin{aligned}
G(Tx,Ty,Ty) &\leq a_1G(x,Tx,Tx)+a_2G(y,Ty,Ty)+a_3G(y,Ty,Ty)+a_4G(x,y,y)\\
&\leq a_1\max\{G(x,Tx,Tx), G(x,x,Tx)\}+(a_2+a_3)\max\{G(y,Ty,Ty),G(y,y,Ty)\}\\
&\quad +a_4\max\{G(x,y,y),G(y,x,x)\} 
\end{aligned}
$$
Replacing  $y$ by $x$ and $z$ by $y$, (II) reduces to
$$
\begin{aligned}
G(Tx,Tx,Ty) &\leq a_1G(x,Tx,Tx)+a_2G(x,Tx,Tx)+a_3G(y,Ty,Ty)+a_4G(x,x,y)\\
&\leq (a_1+a_2)\max\{G(x,Tx,Tx), G(x,x,Tx)\}+a_3 \max\{G(y,Ty,Ty),G(y,y,Ty)\}\\
&\quad +a_4\max\{G(x,y,y),G(y,x,x)\}.
\end{aligned}
$$
Then, it follows from the  above two inequalities that 
$$
\delta(Tx,Ty)\leq A \max\left\{\delta(x,Tx),\delta(y,Ty),\delta(x,y)\right\}
$$ 
for every $x,y\in X$, where $A=\sum_{i=1}^4 a_i<1$ and $\delta$ is the metric on $X$ defined by \eqref{delt-f}, and \'Ciri\'c's fixed point theorem applies. 
\end{remark}

\begin{example}
Let $\lambda\in \left(0,\frac{1}{4}\right)$ be fixed. We consider the set $X_\lambda=\left\{a,b,c,d\right\}\subset \mathbb{R}$, where 
$$
a=0,\quad b=\frac{2\lambda}{2\lambda-1},\quad c=1,\quad d=2.
$$
We introduce the mapping $T: X\to X$ defined by 
$$
Ta=a,\quad Tb=b,\quad Tc=b,\quad Td=b. 
$$
Consider now the $G$-metric on $X$ given by 
$$
G(u,v,w)=\max\left\{|u-v|,|v-w|,|u-w|\right\},\quad u,v,w\in X. 
$$
Remark that $T(Tc)=Tb=b\neq c$ and $T(Td)=Tb=b\neq d$, which shows that condition (I) of Theorem \ref{T3.7} is satisfied. 

We now show that condition (II) of Theorem \ref{T3.7} is satisfied for all pairwise distinct points $x,y,z\in X$. By symmetry, we only study the cases 
$$
(x,y,z)\in \left\{(a,b,c), (a,b,d), (a,c,d), (b,c,d)\right\}.
$$ 
$\bullet$ The case $(x,y,z)=(a,b,c)$: In this case, we have
$$
G(Tx,Ty,Tz)=G(Ta,Tb,Tc)=G(a,b,b)=|a-b|=-b
$$
and 
$$
\begin{aligned}
&\lambda \left[G(x,Tx,Tx)+G(y,Ty,Ty)+G(z,Tz,Tz)+G(x,y,z)\right]\\
&=\lambda \left[G(a,Ta,Ta)+G(b,Tb,Tb)+G(c,Tc,Tc)+G(a,b,c)\right]\\
&=\lambda\left[G(a,a,a)+G(b,b,b)+G(c,b,b)+G(a,b,c)\right],
\end{aligned}
$$
that is,
$$
\begin{aligned}
&\lambda \left[G(x,Tx,Tx)+G(y,Ty,Ty)+G(z,Tz,Tz)+G(x,y,z)\right]\\
&=\lambda\left[G(c,b,b)+G(a,b,c)\right]\\
&=\lambda \left[|c-b|+\max\left\{|a-b|,|b-c|,|a-c|\right\}\right]\\
&=\lambda \left[\frac{1}{1-2\lambda}+\max\left\{\frac{2\lambda}{1-2\lambda},\frac{1}{1-2\lambda},1\right\}\right]\\
&=\lambda \left[\frac{1}{1-2\lambda}+\frac{1}{1-2\lambda}\right]\\
&=\frac{2\lambda}{1-2\lambda}\\
&=-b.
\end{aligned}
$$
Consequently, we obtain
$$
G(Tx,Ty,Tz)=\lambda\left[G(x,Tx,Tx)+G(y,Ty,Ty)+G(z,Tz,Tz)+G(x,y,z)\right].
$$
$\bullet$ The case $(x,y,z)=(a,b,d)$: In this case, we have
$$
G(Tx,Ty,Tz)=G(Ta,Tb,Td)=G(a,b,b)=|a-b|=-b
$$
and 
$$
\begin{aligned}
&\lambda \left[G(x,Tx,Tx)+G(y,Ty,Ty)+G(z,Tz,Tz)+G(x,y,z)\right]\\
&=\lambda \left[G(a,Ta,Ta)+G(b,Tb,Tb)+G(d,Td,Td)+G(a,b,d)\right]\\
&=\lambda\left[G(a,a,a)+G(b,b,b)+G(d,b,b)+G(a,b,d)\right],
\end{aligned}
$$
that is,
$$
\begin{aligned}
&\lambda \left[G(x,Tx,Tx)+G(y,Ty,Ty)+G(z,Tz,Tz)+G(x,y,z)\right]\\
&=\lambda\left[G(d,b,b)+G(a,b,d)\right]\\
&=\lambda \left[|d-b|+\max\left\{|a-b|,|b-d|,|a-d|\right\}\right]\\
&=\lambda \left[\frac{2-2\lambda}{1-2\lambda}+\max\left\{\frac{2\lambda}{1-2\lambda},\frac{2-2\lambda}{1-2\lambda},2\right\}\right]\\
&=\lambda \left[\frac{2-2\lambda}{1-2\lambda}+\frac{2-2\lambda}{1-2\lambda}\right]\\
&=\frac{4\lambda(1-\lambda)}{1-2\lambda}\\
&=-2(1-\lambda)b.
\end{aligned}
$$
Since $-2(1-\lambda)b\leq -b$, then it holds that
$$
G(Tx,Ty,Tz)\leq \lambda\left[G(x,Tx,Tx)+G(y,Ty,Ty)+G(z,Tz,Tz)+G(x,y,z)\right].
$$
$\bullet$ The case $(x,y,z)=(a,c,d)$: In this case, we have
$$
G(Tx,Ty,Tz)=G(Ta,Tc,Td)=G(a,b,b)=|a-b|=-b
$$
and 
$$
\begin{aligned}
&\lambda \left[G(x,Tx,Tx)+G(y,Ty,Ty)+G(z,Tz,Tz)+G(x,y,z)\right]\\
&=\lambda \left[G(a,Ta,Ta)+G(c,Tc,Tc)+G(d,Td,Td)+G(a,c,d)\right]\\
&=\lambda\left[G(a,a,a)+G(c,b,b)+G(d,b,b)+G(a,c,d)\right],
\end{aligned}
$$
that is,
$$
\begin{aligned}
&\lambda \left[G(x,Tx,Tx)+G(y,Ty,Ty)+G(z,Tz,Tz)+G(x,y,z)\right]\\
&=\lambda\left[G(c,b,b)+G(d,b,b)+G(a,c,d)\right]\\
&=\lambda \left[|c-b|+|d-b|+\max\left\{|a-c|,|c-d|,|a-d|\right\}\right]\\
&=\lambda \left[\frac{1}{1-2\lambda}+\frac{2-2\lambda}{1-2\lambda}+\max\left\{1,1,2\right\}\right]\\
&=\frac{\lambda(5-6\lambda)}{1-2\lambda}\\
&=-\frac{(5-6\lambda)}{2}b.
\end{aligned}
$$
Since $-b\leq -\frac{(5-6\lambda)}{2}b$, we also obtain 
$$
G(Tx,Ty,Tz)\leq \lambda\left[G(x,Tx,Tx)+G(y,Ty,Ty)+G(z,Tz,Tz)+G(x,y,z)\right].
$$
$\bullet$ The case $(x,y,z)=(b,c,d)$: In this case, we have
$$
G(Tx,Ty,Tz)=G(Tb,Tc,Td)=G(b,b,b)=0,
$$
which implies immediately that 
$$
G(Tx,Ty,Tz)\leq \lambda\left[G(x,Tx,Tx)+G(y,Ty,Ty)+G(z,Tz,Tz)+G(x,y,z)\right].
$$

Consequently, condition (II) of Theorem \ref{T3.7} holds for all pairwise distinct points $x,y,z\in X$ with $a_1=a_2=a_3=a_4=\lambda$. Notice that since $\lambda\in \left(0,\frac{1}{4}\right)$, then $\sum_{i=1}^4 a_i<1$. Furthermore, we have $\Fix(T)=\{a,b\}$, which confirms the result given by Theorem  \ref{T3.7}.
\end{example}

\section{Applications}\label{sec4}

Some applications of our obtained results are provided in this section. 

\subsection{Mappings contracting perimeters of triangles}
Petrov \cite{Petrov} introduced an interesting class of mappings $T:X \to X$, where $(X,d)$ is a metric space, which can be characterized as mappings contracting perimeters of triangles.
	
\begin{definition}[Petrov \cite{Petrov}]
Let $(X,d)$ be a metric space with $|X|\geq 3$. We shall say that $T:X\to X$ is a mapping
contracting perimeters of triangles on $X$, if there exists $\lambda \in (0,1)$ such that the inequality 
\begin{equation}\label{CT-PET}
d(Tx,Ty) + d(Ty,Tz) + d(Tz,Tx) \leq \lambda [d(x,y)+d(y,z)+d(z,x)],
\end{equation}
holds for all three pairwise distinct points $x,y,z \in X$.
\end{definition}

The following result due to Petrov \cite{Petrov} is an immediate consequence of our Theorem \ref{T3.1}. 

\begin{corollary}\label{CR4.2}
Let $(X,d)$, $|X|\geq 3$, be a complete metric space and let the
mapping $T: X \to X$  satisfies the following two conditions:
\begin{itemize}
\item[{\rm{(I)}}] For all $x\in X$, $T(Tx)\neq x$, provided $Tx\neq x$;
\item[{\rm{(II)}}]	$T$ is a mapping contracting perimeters of triangles on $X$.
\end{itemize}
Then $\Fix(T)\neq\emptyset$ and $|\Fix(T)|\leq 2$. 
\end{corollary}

\begin{proof}
Consider the $G$-metric on $X$ defined by \eqref{G-d}. By Lemma \ref{L2.1}, since $(X,d)$ is a complete metric space, then $(X,G)$ is a complete $G$-metric space. Furthermore, \eqref{CT-PET} is equivalent to 
$$
G(Tx,Ty,Tz)\leq \lambda G(x,y,z)
$$	
for all three pairwise distinct points $x,y,z \in X$. Then, applying Theorem \ref{T3.1}, the desired result follows. 
\end{proof}

Further  results related to mappings contracting perimeters of triangles can be found in \cite{PA,Petrov2,PO1}.

\subsection{Further applications}

Further interesting consequences can be deduced from our obtained results. Let $(X,d)$ be a metric space with $|X|\geq 3$. We introduce the class of mappings $T:X\to X$ satisfying 
the inequality
\begin{equation}\label{Class1}
\max\left\{d(Tx,Ty),d(Ty,Tz),d(Tz,Tx)\right\}\leq \lambda \max\left\{d(x,y),d(y,z),d(x,z)\right\}
\end{equation}
for all three pairwise distinct points $x,y,z \in X$, where $\lambda\in (0,1)$ is a constant. 

\begin{corollary}\label{CR4.3}
Let $(X,d)$, $|X|\geq 3$, be a complete metric space and let the
mapping $T: X \to X$  satisfies the following two conditions:
\begin{itemize}
\item[{\rm{(I)}}] For all $x\in X$, $T(Tx)\neq x$, provided $Tx\neq x$;
\item[{\rm{(II)}}]	\eqref{Class1} holds for all three pairwise distinct points $x,y,z \in X$.
\end{itemize}
Then $\Fix(T)\neq\emptyset$ and $|\Fix(T)|\leq 2$. 
\end{corollary}

\begin{proof}
Consider the $G$-metric on $X$ defined by \eqref{G2-d}.   By Lemma \ref{L2.2}, since $(X,d)$ is a complete metric space, then $(X,G)$ is a complete $G$-metric space. Furthermore, \eqref{Class1} is equivalent to 
$$
G(Tx,Ty,Tz)\leq \lambda G(x,y,z)
$$	
for all three pairwise distinct points $x,y,z \in X$. 	Then, applying Theorem \ref{T3.1}, the desired result follows. 
\end{proof}

We now introduce the class of mappings $T:X\to X$ satisfying the inequality
\begin{equation}\label{CONT-APP3}
d(Tx,Ty)+d(Ty,Tz)+d(Tx,Tz)\leq \lambda \left[d(x,Tx)+d(y,Ty)+d(z,Tz)\right]
\end{equation}
for all three pairwise distinct points $x,y,z \in X$, where $\lambda\in \left(0,\frac{1}{3}\right)$ is a constant.

\begin{corollary}\label{CR4.4}
Let $(X,d)$, $|X|\geq 3$, be a complete metric space and let the
mapping $T: X \to X$  satisfies the following two conditions:
\begin{itemize}
\item[{\rm{(I)}}] For all $x\in X$, $T(Tx)\neq x$, provided $Tx\neq x$;
\item[{\rm{(II)}}]	\eqref{CONT-APP3} holds for all three pairwise distinct points $x,y,z \in X$.
\end{itemize}
Then $\Fix(T)\neq\emptyset$ and $|\Fix(T)|\leq 2$. 
\end{corollary}

\begin{proof}
Consider the $G$-metric defined by \eqref{G-d}. By Lemma \ref{L2.1}, since $(X,d)$ is a complete metric space, then $(X,G)$ is a complete $G$-metric space. Moreover, (\ref{CONT-APP3}) is equivalent to 
$$
G(Tx,Ty,Tz)\leq 
\dfrac{\lambda}{2} [G(x,Tx,Tx)+ G(y,Ty,Ty)+G(z,Tz,Tz)]
$$	
for all three pairwise distinct points $x,y,z \in X$. 	Then, since $\lambda< \dfrac13$ applying Theorem \ref{T3.4}, the desired result follows.
\end{proof}

Consider now the class of mappings $T:X\to X$ satisfying the inequality
\begin{equation}\label{CONT-APP4}
\begin{aligned}
&d(Tx,Ty)+d(Ty,Tz)+d(Tx,Tz)\\
&\leq a_1d(x,Tx)+a_2d(y,Ty)+a_3d(z,Tz)+a_4\left(d(x,y)+d(y,z)+d(x,z)\right)
\end{aligned}
\end{equation}
for all three pairwise distinct points $x,y,z \in X$, where $a_i$ ($i=1,2,3,4$) are nonnegative constants with $0<\sum_{i=1}^4a_i <1$.

\begin{corollary}\label{CR4.5}
Let $(X,d)$, $|X|\geq 3$, be a complete metric space and let the
mapping $T: X \to X$  satisfies the following two conditions:
\begin{itemize}
\item[{\rm{(I)}}] For all $x\in X$, $T(Tx)\neq x$, provided $Tx\neq x$;
\item[{\rm{(II)}}]	\eqref{CONT-APP4} holds for all three pairwise distinct points $x,y,z \in X$.
\end{itemize}
Then $\Fix(T)\neq\emptyset$ and $|\Fix(T)|\leq 2$. 
\end{corollary}

\begin{proof}
Consider the $G$-metric defined by \eqref{G-d}. By Lemma \ref{L2.1}, since $(X,d)$ is a complete metric space, then $(X,G)$ is a complete $G$-metric space. Moreover, (\ref{CONT-APP4}) is equivalent to 
$$
G(Tx,Ty,Tz)\leq \dfrac{a_1}2G(x,Tx,Tx) + \dfrac{a_2}2 G(y,Ty,Ty) + \dfrac{a_3}2 G(z,Tz,Tz) + a_4 G(x,y,z) 
$$	
for all three pairwise distinct points $x,y,z \in X$. 	Then, since $\sum_{i=1}^4a_i<1$, applying Theorem \ref{T3.7}, the desired result follows.
\end{proof}

We finally consider the class of mappings $T:X\to X$ satisfying the inequality
\begin{equation}\label{CONT-APP5}
\begin{aligned}
&\max\left\{d(Tx,Ty),d(Ty,Tz),d(Tx,Tz)\right\}\\
&\leq a_1d(x,Tx)+a_2d(y,Ty)+a_3d(z,Tz)+a_4\max\left\{d(x,y),d(y,z),d(x,z)\right\}
\end{aligned}
\end{equation}
for all three pairwise distinct points $x,y,z \in X$, where $a_i$ ($i=1,2,3,4$) are nonnegative constants with $0<\sum_{i=1}^4a_i <1$.

\begin{corollary}\label{CR4.5}
Let $(X,d)$, $|X|\geq 3$, be a complete metric space and let the
mapping $T: X \to X$  satisfies the following two conditions:
\begin{itemize}
\item[{\rm{(I)}}] For all $x\in X$, $T(Tx)\neq x$, provided $Tx\neq x$;
\item[{\rm{(II)}}]	\eqref{CONT-APP5} holds for all three pairwise distinct points $x,y,z \in X$.
\end{itemize}
Then $\Fix(T)\neq\emptyset$ and $|\Fix(T)|\leq 2$. 
\end{corollary}

\begin{proof}
Consider the $G$-metric defined by \eqref{G2-d}. By Lemma \ref{L2.2}, since $(X,d)$ is a complete metric space, then $(X,G)$ is a complete $G$-metric space. Moreover, (\ref{CONT-APP5}) is equivalent to 
$$
G(Tx,Ty,Tz)\leq a_1G(x,Tx,Tx) + a_2 G(y,Ty,Ty) + a_3 G(z,Tz,Tz) + a_4 G(x,y,z)
$$	
for all three pairwise distinct points $x,y,z \in X$. 	Then, applying Theorem \ref{T3.7}, the desired result follows.
\end{proof}

\subsection*{Data availability statement}
No new data were created or analyzed in this study.

\subsection*{Conflict of interest}
The authors declare that they have no competing interests.

\subsection*{Funding}
The third author  is supported by Researchers Supporting Project number (RSP2024R4), King Saud University, Riyadh, Saudi Arabia.

\subsection*{Authors' contributions}
All authors contributed equally and significantly in writing this article. All authors read and approved the final manuscript.

\end{document}